\documentclass[11pt, a4paper]{amsart}

\title[Bounding mld of general arrangement varieties]{Bounding minimal log discrepancies of general arrangement varieties}
\author{Leandro Meier}
\date{\today}
\address{Leandro Meier, Mathematisches Institut, Universit\"at Jena, Inselplatz 5, 07743 Jena, Deutschland.}
\email{\href{mailto:leandro.meier@uni-jena.de}{leandro.meier@uni-jena.de}}

\usepackage{amsmath}
\usepackage{amsfonts}
\usepackage{amssymb}
\usepackage{amsthm}
\usepackage{enumerate}
\usepackage{booktabs}
\usepackage{graphicx}
\usepackage{tikz-cd}
\usepackage[hidelinks]{hyperref}
\usepackage[toc, page]{appendix}

\newtheorem{theorem}{Theorem}[section]
\newtheorem{proposition}[theorem]{Proposition}
\newtheorem{lemma}[theorem]{Lemma}

\newtheorem{remark}[theorem]{Remark}
\newtheorem{conjecture}[theorem]{Conjecture}

\newtheorem*{theorem*}{Theorem}
\newtheorem*{conjecture*}{Conjecture}

\theoremstyle{definition} 
\newtheorem{definition}[theorem]{Definition}
\newtheorem{example}[theorem]{Example}

\DeclareMathOperator{\cadiv}{CaDiv}
\DeclareMathOperator{\coeff}{coeff}
\DeclareMathOperator{\Cl}{Cl}
\DeclareMathOperator{\codim}{codim}
\DeclareMathOperator{\Cone}{Cone}
\let\d\relax
\DeclareMathOperator{\d}{discrep}
\DeclareMathOperator{\ex}{Ex}
\DeclareMathOperator{\face}{face}
\DeclareMathOperator{\Hom}{Hom}
\DeclareMathOperator{\img}{im}
\let\int\relax
\DeclareMathOperator{\int}{int}
\DeclareMathOperator{\Loc}{Loc}
\DeclareMathOperator{\mld}{mld}
\DeclareMathOperator{\Pol}{Pol}

\DeclareMathOperator{\Spec}{Spec}

\DeclareMathOperator{\supp}{Supp}
\DeclareMathOperator{\tail}{tail}
\DeclareMathOperator{\TV}{TV}
\newcommand{\sslash}{\mathbin{/\mkern-6mu/}}
\newcommand{\nocontentsline}[3]{}
\newcommand{\tocless}[2]{\bgroup\let\addcontentsline=\nocontentsline#1{#2}\egroup}

\begin{document}
\begin{abstract}
  The minimal log discrepancy is an invariant of singularities that plays an important role in the birational classification of algebraic varieties. Shokurov conjectured that the minimal log discrepancy can always be bounded from above in terms of the dimension of the variety. We prove this conjecture for general arrangement varieties, a particular class of $T$-varieties, adding to previous results on this conjecture which include threefolds, toric varieties, and local complete intersection varieties.
\end{abstract}

\maketitle

\tableofcontents

\pagebreak
\section{Introduction}
The minimal log discrepancy is an important invariant in the birational classification of varieties. It has uses in the Minimal Model Program (MMP). Even when run on smooth varieties, to run the MMP successfully in general, one needs to allow varieties with certain mild singularities, and the minimal log discrepancy provides a way of characterizing a minimal class of singularities which need to be allowed.\\
The minimal log discrepancy was introduced in 1988 by Shokurov in \cite{Shok1988OpenProb}, along with the following conjecture regarding its boundedness. For an introduction to these concepts, please see Section \ref{sec:prelims}. 
\begin{conjecture*}[Shokurov, 1988]\label{conj:shok}
  Let $X$ be a normal $\mathbb{Q}$-Gorenstein variety, let $x \in X$ be a closed point. Then $\mld_{x}\left( X \right) \leq \dim \left( X \right)$.
\end{conjecture*}
This conjecture has been proven in dimension at most $3$ by Kawamata in \cite{Kaw1993App} and for toric varieties by Borisov in \cite{Bor1993TorMinDisc}. In \cite{EiMu2004InversionOfAdj}, Ein and Musta\c{t}\u{a} proved the stronger lower semi-continuity conjecture for minimal log discrepancies, from which Shokurov's conjecture follows, for local complete intersection varieties.\\
More recently, in \cite{NaSh2023InvAdjQuotS} Nakamura and Shibata showed that the lower semi-continuity conjecture holds for local complete intersections in quotient varieties, and in \cite{LiZh2025MldOrbif}, Li and Zhou proved Shokurov's conjecture for isolated Fano cone singularities of arbitrary complexity.\\
The goal of article is to prove Shokurov's conjecture for general arrangement varieties, which are a class of $T$-varieties of arbitrary complexity satisfying some conditions that will be made explicit in Section \ref{sec:T-vars}. $T$-varieties are a generalisation of toric varieties, in the sense that they are varieties with an effective torus action, however there is not usually a dense orbit for this action. This enlarges the class of varieties one can consider, while still providing a combinatorial language due to the torus action.\\
The article is organised as follows: in Section \ref{sec:prelims}, we introduce $T$-varieties, revise the basic knowledge of birational geometry needed to understand the article and finally, we introduce the minimal log discrepancy. In Section \ref{sec:mld-bound}, we prove Shokurov's conjecture for general arrangement varieties as well as some auxiliary results. The key ingredient to the proof of our main result is that a general arrangement variety, under additional conditions that will be made explicit throughout Section \ref{sec:mld-bound}, can be presented as a quotient of a complete intersection by a diagonalizable group, due to the results on Cox rings by Hausen, Hische, and Wrobel in \cite{HaHiWr2019OnTorusAct}, and that Shokurov's conjecture holds for local complete intersection varieties, which follows from the more general result in \cite{EiMu2004InversionOfAdj}. In the appendix, we introduce the language of polyhedral divisors, which provides a combinatorial approach to general arrangement varieties and allows us to tie up a few loose ends from Section \ref{sec:mld-bound}. However, most of the arguments in this article can be understood without any knowledge of polyhedral divisors.
\tocless{\section*{Acknowledgements}} I would like to thank Joaqu\'{i}n Moraga for his helpful comments about the behavior of minimal log discrepancies under quotients by finite groups and about the connections between $T$-equivariant log resolutions and minimal log discrepancy. I would like to thank Hendrik S\"{u}ss for many helpful discussions, comments, and questions answered on $T$-varieties and minimal log discrepancies in general.

\section{Preliminaries}\label{sec:prelims}
Unless otherwise stated, all varieties will be considered over $\mathbb{C}$ throughout this text.
\subsection{T-Varieties and general arrangement varieties}\label{sec:T-vars}
$T$-varieties are class of varieties generalizing toric varieties by not requiring the torus action to have a dense orbit, hence allowing for a wider class of varieties. At the same time, the presence of an effective torus action still provides some combinatorial tools to study these varieties. 
The difference in dimensions of the variety and the acting torus, which is often strictly larger than zero, is called the complexity of the $T$-variety.
In the following, we give a brief introduction of $T$-varieties, defining just the notions needed in order to understand the upcoming sections of this article. For a more detailed introduction to this topic, we recommend the survey article \cite{AIOPSV2012GeomOfTVar} as a starting point.
\begin{definition}
  A \emph{$T$-variety} is a normal variety $X$ with an effective action of a torus $T \cong \left( \mathbb{C}^{*} \right)^{k}$ for some $k \leq \dim \left( X \right)$. We call the difference $\dim \left( X \right) - k$ the \emph{complexity} of the torus action.
\end{definition}
\begin{remark}
  A $T$-variety of complexity zero is nothing but a toric variety. 
\end{remark}
Throughout this article, in particular in the proofs, we will often reduce to $T$-varieties with a so-called good torus action, since there are better tools to work with in that case. The resulting varieties are called cone singularities.
\begin{definition}
  Let $X$ be an affine $T$-variety whose torus action has a unique fixed point $x_{0}\in X$ such that $x_{0}$ lies in the closure of every orbit, then we call $\left( X,x_{0} \right)$ a \emph{cone singularity}.
\end{definition}
We may also refer to the cone singularity as just $X$, since the fixed point is unique.
\begin{remark}\label{rem:cone}
  Equivalently, a cone singularity is an affine  $T$-variety $X$ which has an equivariant embedding into an affine space on which the torus $T$ acts with positive weights on the coordinates.
\end{remark}
\begin{example}\
  \begin{enumerate}[(i)]
    \item Every toric variety can be understood as a $T$-variety of complexity greater than zero as well, by restricting to an effective action of a smaller subtorus. For more on these so-called toric downgrades, we recommend the article \cite{AlHa2006PolyAndTorus} as a starting point.
    \item For an example of a $T$-variety that is not a toric downgrade, consider the affine variety $X$ defined as the vanishing locus of the polynomial $x^{2} + y^{3}+ z^{3}$ in $\mathbb{A}^{3}$. We can define an action of the one-torus $\mathbb{C}^{*}$ by specifying $t. \left( x,y,z \right) := \left( t^{3}x, t^{2}y, t^{2}z \right)$, this defines an effective action of $\mathbb{C}^{*}$ on $X$, giving $X$ the structure of a $T$-variety of complexity one.\\
    Since the weights used to define the torus action are positive, we see directly that $X$ is a cone singularity.
  \end{enumerate}
\end{example}
We now introduce some notions of quotients in algebraic geometry that are needed in order to define general arrangement varieties at this point and that we will use in the proofs of Section \ref{sec:mld-bound} as well.\\
Let $X$ be a variety and $G$ a group acting on $X$. We call a morphism $q \colon X \to Y$ \emph{affine} if $q^{-1}\left( V \right) \subseteq X$ is an affine variety whenever $V \subseteq Y$ is an affine open subset.\\
A \emph{good quotient} of $X$ by $G$ is a surjective affine $G$-invariant morphism $q \colon X \to Y$ such that the pullback $q^{*} \colon \mathcal{O}_{Y} \to q_{*}\left( \mathcal{O}_{X} \right)^{G}$ is an isomorphism. In a good quotient $q \colon X \to Y$, $G$-orbits on $X$ and points in $Y$ are not necessarily in a bijective correspondence. Whenever there is a bijective correspondence between these two sets, we call $q$ a \emph{geometric quotient}. We will denote a good quotient by $X \sslash G$ and a geometric quotient by $X/G$.\\
A \emph{rational quotient} of $X$ by $G$ is a $G$-invariant dominant rational map $\pi \colon X \dashrightarrow Y$ such that $\mathcal{O}_{X}^{G} =  \pi^{*}\mathcal{O}_{Y}$ holds. A \emph{representative} of a rational quotient is a surjective morphism $\psi  \colon W \to V$ from a non-empty $G$-invariant open subset $W \subseteq X$ to an open subset $V \subseteq Y$, representing $\pi$. Rosenlicht's theorem (see for instance \cite{Ros1956BasicThms}) guarantees the existence of rational quotients for connected groups $G$.\\
Let $X$ be a $T$-variety, we denote by $X_{0}$ the open subset of points in $X$ that have finite isotropy group under the torus action.
\begin{definition}[Definition 3.12 in \cite{HaHiWr2019OnTorusAct}]
  Let $X$ be a $T$-variety. A \emph{maximal orbit quotient} for $X$ is a rational quotient $\pi \colon X \dashrightarrow Y$ which admits a representative $\psi \colon W \to V$ and prime divisors $C_{0}, \dots, C_{r}$ on $Y$ such that the following hold:
  \begin{enumerate}[(i)]
    \item We have $W \subseteq X_{0}$ and the complements $X_{0} \setminus W \subseteq X_{0}$ and $Y \setminus V \subseteq Y$ are of codimension at least two,
    \item for every $i \in \left\{0, \dots, r\right\}$, the inverse image $\psi^{-1}\left( C_{i} \right) \subseteq W$ is a union of prime divisors $D_{i1}, \dots, D_{in_{i}}$ in $W$,
    \item all $T$-invariant prime divisors on $X_{0}$ with nontrivial generic isotropy group occur among the $D_{ij}$,
    \item every sequence $J = \left( j_{0}, \dots, j_{r} \right)$ with $1 \leq j_{i}\leq n_{i}$ defines a geometric quotient $\psi  \colon W_{J} \to V$ for the torus action, where $W_{J}:= W \setminus \bigcup_{j\ne j_{i}}D_{ij}$.
  \end{enumerate}
  We call $C_{0}, \dots, C_{r} \subseteq Y$ a collection of \emph{doubling divisors} for $\pi \colon X \dashrightarrow Y$.
\end{definition}

\begin{definition}
  A \emph{general arrangement variety} is a  $T$-variety $X$ of complexity $d$ with a maximal orbit quotient $\pi \colon X \dashrightarrow \mathbb{P}^{d} $ such that the doubling divisors $C_{0}, \dots, C_{r}$ form a general hyperplane arrangement in $\mathbb{P}^{d}$.
\end{definition}
\begin{example}
  \begin{enumerate}[(i)]
    \item Any rational complexity one $T$-variety is a general arrangement variety, in essence due to the fact that any point configuration on $\mathbb{P}^{1}$ is trivially a general hyperplane arrangement. For more details, see Remark 6.2 in \cite{HaHiWr2019OnTorusAct}.
  \item We consider the variety $X$, defined as the affine cone over the hypersurface $Q$, where $Q$ is the vanishing locus of the polynomial $\sum_{i=0}^{n} x_{i}y_{i}$ inside $\mathbb{P}^{n}\times \mathbb{P}^{n}$ and embedded into $\mathbb{P}^{\left( n+1 \right)^{2}-1}$ via the Segre embedding. Thus $X$ is an affine singular variety of dimension $2n$.\\
    The minimal codimension of a torus which can act effectively on $Q$ or $X$ is the same, since $\dim X = \dim Q +1$ but on $X$ there is additionally an action by a one-torus with weights all one. It is simpler to deduce this minimal codimension on $Q$. A torus $T$ acting in effectively on $Q$ with weights $a_{0}, \dots, a_{n}$, $b_{0}, \dots, b_{n}$ needs to satisfy the $n$ linearly independent conditions $a_{0} + b_{0} = a_{i}+b_{i}$ for $i\in \left\{1,\dots,n\right\}$. Hence the minimal codimension in $Q$ for a torus $T$ acting effectively on $Q$ is $\dim Q - n = n-1$. Hence for all $n \geq 3$, one can only consider $Q$ resp. $X$ as a $T$-variety of complexity at least two. \\
    Moreover, an effective action of an $\left( n+1 \right)$-dimensional torus $T$ on $X$ is realised by the following matrix of size $\left( n+1 \right) \times \left( 2n+2 \right)$, where each column is a vector defining a multidegree in $\mathbb{Z}^{n+1}$ on the corresponding coordinate in $\mathbb{P}^{n}\times\mathbb{P}^{n}$:
    \[
      \left(\begin{array}{ c ccc | c ccc }
          0 &  1 & \dots & 0 & 0 &-1 & \dots & 0 \\
          \vdots & \vdots & \ddots & \vdots & \vdots & \vdots & \ddots & \vdots \\
          0 & 0 & \dots & 1 & 0 & 0 & \dots & -1 \\
          1 & 1 & 1 & 1 & 1 & 1 & 1 & 1
      \end{array}\right).
    \]
    There is a rational quotient map $\pi \colon X \dashrightarrow \mathbb{P}^{n-1}$ given as follows: on the open subset $X_{0} = X \setminus V \left( x_{0} y_{0} \right)$, every point has finite stabilizer under the torus action. We let $W=X_{0}$ and define the rational quotient $\pi  \colon X \dashrightarrow \mathbb{P}^{n-1}$ on $W$ by the representative $\psi \colon W \to \mathbb{P}^{n-1}$ as 
    \[
      \psi  \left([x_{0} : \dots : x_{n}],[y_{0} : \dots : y_{n}]   \right) = [x_{1}y_{1} : \dots : x_{n}y_{n}].
    \]
    Note that the dimension of the codomain of $\psi$ reflects the complexity of the torus action. We denote the homogeneous coordinates on $\mathbb{P}^{n-1}$ by $[z_{1} : \dots : z_{n}]$.\\
    In Section 3.1 of \cite{Cab2019KESymGAV}, it is shown that the doubling divisors of this rational quotient are precisely the coordinate hyperplanes $H_{i} = V \left( z_{i} \right)$ as well as the hyperplane $H_{0 }= V \left( \sum_{i=0}^{n} z_{i} \right)$ in $\mathbb{P}^{n-1}$. Moreover, the data $\pi$ and $H_{0}, \dots, H_{n}$ form a maximal orbit quotient. These hyperplanes form a general hyperplane arrangement in $\mathbb{P}^{n-1}$, hence $X$ is a general arrangement variety. 
  \end{enumerate}
\end{example}
\subsection{Background on divisors}
In this subsection, we recall some notions on divisors and birational geometry that we will use throughout the article. We start by recalling some more well-known facts both to fix the notation and to facilitate reading for those who are not yet familiar with these terms. We will assume that $X$ is a normal variety throughout this subsection.\\
By a \emph{prime (Weil) divisor} $D$ on a normal variety $X$, we denote an irreducible closed subvariety $D$ of codimension one in $X$. 
We consider the free abelian group generated by the prime divisors on $X$, the elements of this group are called \emph{Weil divisors} or just \emph{divisors} if no confusion is likely to arise from this. 
We call a divisor $D$ on $X$ \emph{principal} if it is of the form $\sum_{D \subseteq X}^{} \operatorname{ord}_{D}\left( f \right)$, where $f\in \mathbb{C}^{*}\left( X \right)$, the sum runs over all prime divisors $D \subseteq X$ and $\operatorname{ord}_{D}\left( f \right)$ denotes the order of vanishing of $f$ along $D$. One can show that this is indeed a finite sum. 
We define the \emph{class group} of $X$ as the quotient of the group of Weil divisors on $X$ by the subgroup generated by all principal divisors on $X$, we will denote it by $\Cl \left( X \right)$. \\
We call a Weil divisor $D$ on $X$ \emph{Cartier} if there exists an open cover $\left\{ U_{i}\right\}_{i\in I}$ of $X$ as well as rational functions $f_{i}\in \mathbb{C}^{*} \left( U_{i} \right)$ for each $i\in I$ such that locally on each of the $U_{i}$, $D$ is given as the principal divisor associated to $f_{i}$ and moreover, $f_{i}/f_{j} $ lies in $\mathcal{O}^{*}\left( U_{i}\cap U_{j} \right)$ for all $i,j \in I$.\\
In order to define the minimal log discrepancy, we will need to work with $\mathbb{Q}$-divisors as well. A \emph{$\mathbb{Q}$-divisor} on a normal variety $X$ is a finite formal sum $\sum_{i\in I}^{} r_{i}D_{i}$, where the $D_{i}$ are prime divisors on $X$ and the coefficients $r_{i}$ are rational numbers.
\begin{definition}
  Let $X$ be a normal variety, we say that a Weil divisor $D$ on $X$ is \emph{$\mathbb{Q}$-Cartier} if there exists some integer $m > 0$ such that $m \cdot D$ is a Cartier divisor.
\end{definition}
\begin{remark}
  In contrast to an arbitrary Weil divisor, if $D$ is a $\mathbb{Q}$-Cartier divisor on $X$, then we can be sure that its pullback along any birational morphism into $X$ is well defined: let $m > 0$ be such that $m\cdot D$ is Cartier and let $f \colon Y \to X$ a birational morphism, then we define $f^{*}\left( D \right) := \frac{1}{m}f^{*}\left( mD \right)$.
\end{remark}
\begin{definition}
  Let $X$ be a normal variety, we say that $X$ is 
  \begin{enumerate}[(i)]
    \item \emph{$\mathbb{Q}$-Gorenstein} if its canonical divisor $K_{X}$ is $\mathbb{Q}$-Cartier,
    \item \emph{$\mathbb{Q}$-factorial} if every Weil divisor on $X$ is $\mathbb{Q}$-Cartier.
  \end{enumerate}
\end{definition}
\begin{definition}
  Let $f \colon Y \to X$ be a birational morphism between normal varieties $Y$ and $X$. The set of those points $y \in Y$ for which the inverse rational map $f^{-1} \colon X \dashrightarrow Y$ is not defined at $f\left( y \right)$ is called the \emph{exceptional locus of $f$} and denoted by $\ex \left( f \right)$.
\end{definition}
\begin{definition}
  We say that a divisor  $D = \sum_{i \in I} a_{i}D_{i}$ on a smooth variety $X$ has \emph{simple normal crossings (snc)} if all the prime divisors $D_{i}$ are distinct and smooth and moreover, at every point $x \in X$, the following equality holds:
  \[
    \codim \left( \bigcap_{x \in D_{i}}^{} T_{p}\left( D_{i} \right) \right) = | \left\{i \in I \mid p \in D_{i}\right\} | .
  \]
\end{definition}
\subsection{The minimal log discrepancy}\label{sec:mld}
In the study of singularities and log discrepancies, it can be useful to consider a log pair, which is a tuple $\left( X, \Delta \right)$ consisting of a variety $X$ and a $\mathbb{Q}$-divisor $\Delta$, subject to some conditions.
\begin{definition}
  Let $X$ be a normal variety and $\Delta = \sum_{i\in I}^{} a_{i}D_{i}$ a (not necessarily effective) $\mathbb{Q}$-divisor on $X$ such that $K_{X} + \Delta$ is $\mathbb{Q}$-Cartier. We call the tuple $\left( X, \Delta \right)$ a \emph{log pair}. If moreover $a_{i} \leq 1$ holds for all $i \in I$, we call $\Delta$ a \emph{boundary} divisor.
\end{definition}
For instance, if the variety $X$ is not $ \mathbb{Q} $-Gorenstein, a divisor $\Delta$ as above might be found so that $K_{X} + \Delta$ is $\mathbb{Q}$-Cartier, thus allowing the log discrepancy and related concepts to be defined for a wider class of varieties. If $\Delta $ is trivial, we will omit it to simplify the notation.
\begin{definition}\label{def:ld}
  Let $\left( X, \Delta \right) $ be a log pair. Let $Y$ be a normal variety, let $f \colon Y \to X$ be a proper birational morphism, and $E \subseteq Y$ a prime divisor. The \emph{log discrepancy} of $\left( X,\Delta \right)$ at $E$ is the rational number
  \[
    a_{E}\left( X,\Delta \right) := \coeff_{E}\left( K_{Y} - f^{*}\left( K_{X} + \Delta\right) \right) + 1,
  \]
  where $K_{Y}$ and $K_{X}$ denote a choice of canonical divisor on $Y$ resp. $X$ such that $f_{*} K_{Y} = K_{X}$ holds. We also call $E$ a \emph{divisor over $X$}.
\end{definition}
  Although it may look as if the log discrepancy depends on the choice of $Y$ and $f$, it is in fact only dependent on the divisor $E$. 
  To see why this is true, note that each prime divisor $E \subseteq Y$ as above induces a valuation $v_{E}$ on $\mathbb{C}\left( Y \right)$, the vanishing order along $E$, and that $a_{E} \left( X, \Delta \right)$ depends only on this valuation.
\begin{remark}
  Let the notation be as in Definition \ref{def:ld}. By the \emph{discrepancy of $E$ over $\left( X,\Delta\right)$}, we denote the rational number 
  \[
    \d_{E}\left( X,\Delta \right) := \coeff_{E}\left( K_{Y} - f^{*} \left( K_{X} +\Delta\right)\right),
  \]
  which is one less than the log discrepancy. We will use it mainly for technical purposes in proofs throughout this text.
\end{remark}
\begin{definition}
  Let $\left( X ,\Delta\right)$ be a log pair, where $\Delta = \sum_{i\in I}^{} a_{i}D_{i}$. We call $X$ \emph{log canonical} if $a_{E}\left( X, \Delta \right) \geq 0$ holds for every exceptional prime divisor $E$ over $X$. We call $X$ \emph{Kawamata log terminal} (or klt for short) if $a_{E}\left( X, \Delta \right) > 0$ holds for every exceptional prime divisor $E$ over $X$ and moreover, all $a_{i} < 1$.
\end{definition}
\begin{remark}
  In some texts, a log pair $\left( X, \Delta \right)$ is only called log canonical resp. klt if additionally, the divisor $\Delta$ is effective, and otherwise it is called sub log canonical resp. sub klt. We will not make this distinction here. Similarly, a boundary divisor might be required to be effective in some sources, and otherwise be called a sub boundary divisor, but we will not make this distinction either.
\end{remark}
One particular use case of boundary divisors is to keep track of log discrepancies as one performs birational modifications on a variety. Let $\left( X, \Delta \right)$ be a log canonical pair, where $X$ is not necessarily smooth, and let $ f \colon X' \to X$ be a resolution of $X$. By defining the \emph{discrepancy divisor} 
  \[
    \Delta' = \sum_{E}^{} -\d_{E}\left( X, \Delta \right)\cdot E ,
  \]
where the sum runs over all prime divisors in the support of $\ex \left( f \right)$, we obtain a boundary divisor $\Delta'$ on $X'$ such that $a_{E'}\left( X , \Delta \right) = a_{E'}\left( X', \Delta' \right)$ for any prime divisor $E'$ over $X$. \\
A useful tool for computing log discrepancies in concrete scenarios is given by Lemma 2.29 in \cite{KoMo1998BirGeom}, a special case of which we recall in the following example.
\begin{example}\label{blowup-formula}
  Let $\left( X, \Delta \right)$ be a log pair, where $\Delta = \sum_{i\in I}^{} a_{i}D_{i}$ such that the $D_{i}$ are distinct prime divisors. Let $Z \subseteq X$ be a smooth closed irreducible subset of codimension $k \geq 2$ and consider the blow-up $f \colon \tilde{X} \to X$ of $X$ in $Z$. Let $E$ be the exceptional divisor of the blow-up, which is always a prime divisor in this setting. Then the \emph{blow-up formula} for log discrepancies states that $a_{E} \left( X,\Delta \right) = k - \sum_{i \in I}^{}a_{i} $, where $I$ denotes the set of those $i$ for which $Z \subseteq D_{i}$ holds.
\end{example}
\begin{definition}
  Let $\left( X, \Delta \right)$ be a log pair and $W$ a nonempty closed subset of $X$. The \emph{minimal log discrepancy} of $\left( X, \Delta \right)$ at $W$ is the rational number
  \[
    \mld_{W}\left(X, \Delta \right) := \min \left\{a_{E}\left( X, \Delta \right) \mid f \colon E \subseteq Y \to X, c_{X}\left( E \right) \subseteq W \right\}.
  \]
  The minimum is taken over all $Y$ and $f$ as in Definition \ref{def:ld} above, with $E$ running over all prime divisors over $X$ that satisfy $c_{X} \left( E \right) \subseteq W$, where $c_{X}\left( E \right):= f \left( E \right)$ denotes the \emph{center} of the divisor $E$ on $X$.
\end{definition}
\begin{example}
  From the blow-up formula, it immediately follows that if $X$ is a smooth variety and $W \subseteq  X$ a smooth closed subset of codimension $k$, then $\mld_{W}\left( X \right) \leq k$, and in fact, equality holds too. 
\end{example}
\begin{definition}
  A \emph{log resolution} of a log pair $\left( X, \Delta \right)$ is a proper birational morphism $f \colon \tilde{X} \to X$, where $\tilde{X}$ is a smooth variety, satisfying the following properties:
  \begin{enumerate}[(i)]
    \item the smooth locus of $X$ is isomorphic to $\tilde{X}$,
    \item the exceptional locus $\ex \left( f \right)$ is a divisor on $\tilde{X}$, and
    \item $\ex \left( f \right) \cup f^{-1}\left( \supp \left( \Delta \right) \right)$ is an snc divisor. 
  \end{enumerate}
\end{definition}
When working with minimal log discrepancies, it is sometimes useful to restrict to an (affine) open subset of the ambient variety. This is justified by the fact that for any closed subset $W \subseteq X$, the minimal log discrepancy of $X$ at $W$ is invariant under the restriction to open subsets of $X$ containing $W$. This fact seems to be well-known among experts, however, we were unable to find a proof in the literature, which is why we provide one here.
\begin{lemma}\label{lem:opensubset}
  Let $X$ be a $ \mathbb{Q} $-Gorenstein variety, let $W \subseteq X$ be a closed subset and $U \subseteq X$ an open subset such that $W \cap U \ne \emptyset $. Then $\mld_{W}\left( X \right) = \mld_{W}\left( U \right)$.
\end{lemma}
\begin{proof}
  For the inequality $\geq$, it is enough to observe that any exceptional prime divisor over $X$ with center contained in $W$ is also a prime divisor over $U$ whose center is contained in $W \cap U$.\\
  For the other inequality, let $f \colon \tilde{U} \to U$ be a log resolution that extracts an exceptional divisor $E$ with center contained in $W$. There is a complete variety $C$ in which we can embed $\tilde{U}$ as an open subset, by Theorem 4.3 in \cite{Nag1962Imb}. We define a morphism $\overline{f}\colon \tilde{U} \to X \times C$, $ u \mapsto \left( f \left( u \right), u \right)$ and let $\overline{X}$ be the closure of $\img \left( \overline{f} \right)$ in $X \times C$.\\
  Let $g \colon X_{n} \to \overline{X}$ be the normalization of $\overline{X}$, note that $g$ is a birational and proper morphism. Let $\iota$ be the closed embedding of $\overline{X}$ into $X \times C$ and $p_{X} \colon X \times C \to X$ the projection to $X$. Both of these morphisms are proper, and the composition $p_{X} \circ \iota$ is proper and birational. Thus the composition $p_{X} \circ i \circ g$ is a proper birational morphism as well. The following diagram sketches the situation.
  \[
\begin{tikzcd}
                 & E' \subseteq X_{n} \arrow[d, "g"]       &                      \\
  E \subseteq \tilde{U} \arrow[r, "\overline{f}"] & \overline{X} \arrow[r, "\iota", hook] & X \times C \arrow[d, "p_{X}"] \\
                 &                            & X                   
\end{tikzcd}
  \]
  Defining the divisor $E'$ on $X_{n}$ as the birational transform of $E$ under $g^{-1} \circ \overline{f}$, we get an exceptional prime divisor over $X$ with center contained in $W$ and $a_{E}\left( U \right) = a_{E'}\left( X \right)$, finishing the proof.
\end{proof}
\section{Shokurov's conjecture for general arrangement varieties}\label{sec:mld-bound}
The goal of this section is to prove the following theorem stating that Shokurov's conjecture holds for klt general arrangement varieties of arbitrary complexity.
\begin{theorem}\label{thm:main}
  Let $X$ be a klt $\mathbb{Q}$-Gorenstein general arrangement variety of complexity $d$, let $x \in X$ be a closed point. Then $\mld_{x}\left( X\right) \leq \dim \left( X \right)$.
\end{theorem}
The occasional argument in the proof of this theorem uses the language of polyhedral divisors for $T$-varieties, but the bulk of the proof can be understood without knowing anything about polyhedral divisors. Thus for the sake of smoother reading, they are only introduced in Appendix \ref{sec:p-divs}, where the interested reader can read up on details.
\subsection{Minimal log discrepancy under group quotients}
The proof of Theorem \ref{thm:main} uses, among others, the statement that if $Z$ is a variety and $X $ is a geometric quotient of $Z$ by a quasitorus $H$, satisfying some conditions, then roughly speaking, the minimal log discrepancy on $X$ is bounded from above by the minimal log discrepancy on $Z$. By a quasitorus, we denote a group which is a direct product of a torus and a finite abelian group.
\begin{proposition}\label{prop:quotient}
  Let $Z$ be a variety and $H$ a quasitorus acting on $Z$ such that the isotropy group of every point $z\in Z$ is finite and the generic isotropy group of every prime divisor $E \subseteq Z$ is trivial. Suppose that the quotient $q \colon Z \to Z \sslash H$ is geometric and let $X = Z/H$. Let $x \in X$ be an arbitrary point and $C$ an irreducible component of $q^{-1}\left( x \right)$, then $\mld_{x} \left( X \right) $ is bounded from above by $\mld_{C}\left( Z \right)$.
\end{proposition}
Before the proof of this proposition, we state two lemmas relating the minimal log discrepancy of a geometric quotient $Z/ H$ to the minimal log discrepancy of $Z$ in the special cases where the group $H$ is either finite or a torus and the action satisfies some technical conditions.\\
The following Lemma \ref{lem:Masek} is a slight generalisation of Lemma 5 in \cite{Mas2001MinDis} replacing the closed points $y\in Y$ and $y' \in Y'$ by closed subsets $C \subseteq Y$ and $C' \subseteq Y'$. The proof is analogous to the proof of the original statement. 
 \begin{lemma}\label{lem:Masek}
  Let $\varphi  \colon Y' \to Y$ be a finite morphism of normal $ \mathbb{Q} $-Gorenstein  varieties. Assume that $\varphi$ is \'etale in codimension one. Let $C$ be a closed irreducible subset of $Y$ and $C'$ an irreducible component of $\varphi^{-1} \left( C \right)$. Then $\mld_{C}\left( Y \right) \leq \mld_{C'}\left( Y' \right)$.
\end{lemma}
\begin{lemma}\label{lem:torus-q}
  Let $Z$ be a variety and $T$ a torus acting on $Z$ with trivial stabilizers and such that $q \colon Z \to X$ is a geometric quotient. Then the minimal log discrepancy satisfies $\mld_{x} ( X ) = \mld_{q^{-1}\left( x \right)} ( Z )$ 
  for all $x\in X$.
\end{lemma}
\begin{proof}
  Since the torus acts without stabilizers, by Proposition 5.7 in \cite{Dre2004LunaSlice}, $Z \to Z/T$ is a principal $T$-bundle and thus also a Zariski-locally trivial bundle, hence after restricting to an open subset if necessary, we may assume that $Z \cong X \times T$ by $T$-equivariant isomorphism. We let $x \in X$ and identify $q^{-1}\left( x \right)\in Z/T$ with the pair $\left( x,T \right)$ under this isomorphism. By the product formula for minimal log discrepancies, see Remark 2.7 in \cite{Am1999OnMLD}, we get $\mld_{\left( x,T \right)}( Z ) = \mld_{x}\left( X \right) + \mld_{T}\left( T \right) $, and the right-hand-side of this is equal to $\mld_{x}\left( X \right) $, since the torus $T$ being smooth implies that $\mld_{T}\left( T \right) = 0$.
\end{proof}
\begin{proof}[Proof of Proposition \ref{prop:quotient}] 
  We fix a point $x\in X$ and consider its inverse image $q^{-1}\left( x \right)\subseteq Z$ under the quotient morphism $q$. Because the quotient is geometric, $q^{-1}\left( x \right)$ is equal to a closed orbit $H.z$ for some $z\in Z$. It follows from Proposition 5.5 in \cite{Dre2004LunaSlice} that there exists an $H$-invariant open subset $U_{0} \subseteq Z$ containing $H.z$ and such $G_{1}$, the subgroup of $H$ generated by all stabilizers of points in $U_{0}$, is finite. Since $\mld_{x}$ is not affected by restricting to an open neighborhood of $x$, we will consider $\mld_{x}\left( U_{0} \right)$ for the rest of the proof.\\
 We consider the quotient $U_{1}:=U_{0} / G_{1} $ and define an action of the factor group $H/G_{1}$ on $U_{1}$ by specifying $\overline{h}. \left( G_{1}.u \right) := G_{1}.\left( h.u \right)$, where $\overline{h} = hG_{1}$ and $u_{}\in U_{0} $. 
  It is straightforward to check that this is well-defined and that moreover,
  \[
    \mathbb{C}[ U_{0} ]^{H}\cong \big(\mathbb{C}[ U_{0} ]  ^{G_{1}}\big)^{H/G_{1}}
  \]
  holds for the respective invariant rings. Thus $U = U_{0}/H$ is isomorphic to the quotient $U_{1} / \left( H/G_{1} \right)$.\\
  The quotient group $H' = H/G_{1}$ is isomorphic to a product $T \times G_{2}$, where $T$ is a torus and $G_{2}$ a finite abelian group. We apply the same steps as above, considering quotients $U_{2} := U_{1}/G_{2}$ and $U_{3} := U_{2}/T $, noting that $T \cong H'/G_{2}$ acts on $U_{2}$ with trivial stabilizer. By the same arguments as above, $\left( U_{1}/H' \right) \cong  \left( U_{1} /G_{2}\right)/T$, and the left-hand-side of this is isomorphic to $U_{0}/H$.\\
  We denote the respective quotient morphisms by
  \[
      U_{0} \xrightarrow{q_{1}} U_{1} \xrightarrow{q_{2}} U_{2} \xrightarrow{q_{3}} U_{3}
  \]
  and let 
  \begin{itemize}
    \item $C_{2} \subseteq U_{2}$ be an irreducible component of $q_{3}^{-1}\left( x \right)$, 
    \item $C_{1} \subseteq U_{1}$ an irreducible component of $q_{2}^{-1}\left( C_{2} \right)$, 
    \item $C_{0} \subseteq U_{0}$ an irreducible component of $q_{1}^{-1}\left( C_{1} \right)$.
  \end{itemize}
  By Lemma \ref{lem:torus-q}, $\mld_{x}\left( U \right) \leq \mld_{C_{2}}\left( U_{2} \right)$. Since the action of $H$ on $Z$ has trivial generic stabilizer along prime divisors, the same is true for the subgroups $G_{1}$ and $G_{2}$, hence the quotient morphisms $q_{1}$ and $q_{2}$ are finite and \'etale in codimension one. We can apply Lemma \ref{lem:Masek} twice and get $\mld_{C_{2}}\left( U_{2} \right) \leq \mld_{C_{1}}\left( U_{1} \right) \leq \mld_{C_{0}}\left( U_{0} \right)$. Observing that $C_{0}$ is an irreducible component of $q^{-1}\left( x \right)$, the proof is done here.
\end{proof}
\subsection{Proof of the main result}
As an intermediate step in proving the main result, we prove that Shokurov's conjecture holds for general arrangement varieties that are also cone singularities.
\begin{theorem}\label{thm:cone-sing}
  Let $X$ be a klt general arrangement variety of complexity $d$ and assume that $X$ is moreover a cone singularity. Let $x \in X$ be a closed point, then $\mld_{x}\left( X \right) \leq \dim \left( X \right)$.
\end{theorem}
\begin{proof}
  Because $X$ is a cone singularity, the only invertible global sections on $X$ are the constant ones. Since $X$ is a general arrangement variety, it follows from Corollary 3.15 in \cite{PeSu2011TorInvDiv} that the class group of $X$ is finitely generated.\\
  If $X$ is not $ \mathbb{Q} $-factorial, let $g \colon Y \to X$ be the small, $T$-equivariant $ \mathbb{Q} $-factorialization from Lemma \ref{lem:small-Q-fact}, with $Y$ a $T$-variety of the same dimension and complexity as $X$. If $X$ is $ \mathbb{Q} $-factorial, replace $Y$ by $X$ in the following steps.\\
  Since $g$ is small and birational, $\mld_{x}\left( X \right) = \mld_{y}\left( Y \right)$ holds for every $y \in g^{-1}\left( x \right)$ and the divisor class groups $\Cl\left( X \right)$ and $\Cl\left( Y \right)$ are isomorphic.\\
  Applying Construction 1.6.1.3 in \cite{ArDeHaLa2015CoxRings} to $R$ and $Y$, we let $R$ denote the Cox ring of $Y$ and set $\overline{Y}=\Spec R$. Due to the isomorphic class groups, the Cox ring of $Y$ is isomorphic to the Cox ring of $X$. Since $X$ is a general arrangement variety, the Cox ring of $X$ is a complete intersection ring by Theorem 6.5 in \cite{HaHiWr2019OnTorusAct} and then so is $R$.\\
  Let $\hat{Y} \subseteq \overline{Y}$ be an open subset such that $Y$ is isomorphic to the good quotient $\hat{Y}  \sslash H$. The existence of $\hat{Y}$ is guaranteed by the theory of Cox rings, for details on this please consider Section 1.6 of \cite{ArDeHaLa2015CoxRings}.\\
  Since $Y$ is $ \mathbb{Q} $-factorial, Corollary 1.6.2.7 in \cite{ArDeHaLa2015CoxRings} implies that the quotient $\hat{Y} \sslash H$ is geometric and moreover, according to Corollary 1.3.4.8 in \cite{ArDeHaLa2015CoxRings}, $H$ acts with at most finite isotropy groups on $\hat{Y}$. By Proposition 1.6.1.6 in \cite{ArDeHaLa2015CoxRings}, the action of $H$ on $\hat{Y}$ has trivial generic isotropy group on every prime divisor in $\hat{Y}$.\\
  Thus we can apply Proposition \ref{prop:quotient} to $\hat{Y}$, $H$ and an arbitrary point $y\in g^{-1}\left( x \right)\subseteq Y$, to obtain that $\mld_{y}\left( Y \right) \leq \mld_{W}(\hat{Y}) $ holds for any irreducible component $W$ of $q^{-1}\left( y \right) \subseteq \hat{Y}$, where $q \colon \hat{Y} \to Y$ is the quotient morphism. Since $R$ is a complete intersection, the right-hand-side of this is bounded from above by $\codim_{\hat{Y}}\left( W \right)$ by Theorem 1.2 in \cite{EiMu2004InversionOfAdj}, and $\codim_{\hat{Y}}\left( W \right) \leq \dim \left( Y \right) = \dim \left( X \right)$, thus finishing the proof. 
\end{proof}
\begin{lemma}\label{lem:small-Q-fact}
Let $X$ be a klt $\mathbb{Q}$-Gorenstein $T$-variety of complexity $d$. Then there exists a $\mathbb{Q}$-factorial $T$-variety $Y$ of the same complexity and dimension as $X$, and a small $T$-equivariant contraction $g \colon Y \to X$. In particular, the equality $\mld_{x}\left( X \right) = \mld_{y}\left( Y \right)$ holds for every $y \in g^{-1}\left( x \right)$.
\end{lemma}
Here by a \emph{contraction}, we denote a projective surjective birational morphism $g \colon Y \to X$, and we call a birational morphism $g \colon Y \to X $ \emph{small} if its exceptional locus has codimension at least two in $Y$.
\begin{proof}
  This proof is inspired by the proof of Proposition 2.1 in \cite{Kaw2011TowardsBound}.\\
  Let $f \colon \tilde{X} \to X$ be a $T$-equivariant log resolution, which exists by Theorem 8.11 in \cite{Vil2014Equimultiplicity}. Because $X$ is log terminal, there are no $f$-exceptional prime divisors with log discrepancy equal to zero, thus by Proposition 4.4.1 in \cite{Pro2021EquivMMP}, there exists a \emph{small} $T$-equivariant contraction $g \colon Y \to X$, for a $ \mathbb{Q} $-factorial variety $Y$. The equality for the minimal log discrepancies follows directly from the fact that $g$ is small.\\
  Note that the hypothesis of Proposition 4.4.1 in \cite{Pro2021EquivMMP} assumes the group acting on $X$ to be finite, but by §1 of the same article, the statement of the proposition also holds for any algebraic group that is acting algebraically on $X$.\\
  Observe moreover that Proposition 4.4.1 in \cite{Pro2021EquivMMP} states that $Y$ is $T\mathbb{Q}$-factorial, i.e. every $T$-invariant Weil divisor is  $\mathbb{Q}$-Cartier. However, since any Weil divisor on a $T$-variety is linearly equivalent to a $T$-invariant Weil divisor, this implies that $Y$ is $\mathbb{Q}$-factorial in our setup.
\end{proof}
With all preliminary results covered, we are now able to prove the main theorem.
\begin{proof}[Proof of Theorem \ref{thm:main}]
  If $X$ is a cone singularity, the result follows from Theorem \ref{thm:cone-sing}. Thus we consider the case when $X$ is not a cone singularity. By the theory of $T$-varieties, there is a cover of $X$ by $T$-invariant open affine subsets, which are still general arrangement varieties.
  Since the minimal log discrepancy is a local invariant, we may replace $U$ by one of these open subsets that contains $x$ for the remainder of the proof. We will denote this subset by $U$.\\
  If $U$ is a cone singularity, again Theorem \ref{thm:cone-sing} applies. Otherwise, we know that $U$ is an affine general arrangement variety and not a cone singularity. Then there are exactly two possible cases, which we will treat separately. For a proof of this as well as more details on the remaining part of the proof, please consider the appendix, particularly Lemma \ref{lem:not_cone}.\\
  In the first case, $U$ is toroidal, i.e for every $u\in U$, there is a formal neighborhood isomorphic to a formal neighborhood of a point in an affine toric variety. Then we can apply Proposition 6 from \cite{Mas2001MinDis}, which states that the minimal log discrepancy is invariant under analytic isomorphism, to deduce that the result follows from Shokurov's Conjecture for toric varieties, which was proven in \cite{Bor1993TorMinDisc}.\\
  In the second case, $U$ can be embedded $T$-equivariantly as an open subset of a general arrangement variety which is a cone singularity and of the same dimension as $U$. Due to this invariant open embedding, the result follows with Theorem \ref{thm:cone-sing} and Lemma \ref{lem:opensubset}.
\end{proof}
\subsection{Complexity one}
  In the case of complexity one, we are able to prove a slightly more general statement, no longer requiring the hypothesis that the $T$-variety $X$ is klt. 
  \begin{theorem}
    Let $X$ be any $ \mathbb{Q} $-Gorenstein $T$-variety of complexity one. Then Shokurov's conjecture holds for $X$.
  \end{theorem}
  We point out the main observations needed to deduce this from the results in this section.
  \begin{enumerate}[(i)]
  \item \label{obs-i} Every $ \mathbb{Q} $-factorial cone singularity of complexity one is rational and hence in particular a general arrangement variety.
    \item If $X$ is not log canonical, we can prove by a direct argument (without using Cox rings) that for any $x \in X$, $\mld_{x}\left( X \right) \leq \dim \left( X \right)$, see Lemma \ref{lem:not-lc} below.
    \item If $X$ is a $T$-variety that is log canonical but not klt, it is not clear whether the $ \mathbb{Q} $-factorialization $g \colon Y \to X$ from Lemma \ref{lem:small-Q-fact} is a small morphism. Hence it is not possible, with our tools, to know whether the Cox ring of $Y$ satisfies the properties we need from it. This obstacle does not appear in complexity one, due to observation \eqref{obs-i}.
  \end{enumerate}
 The following Lemma seems to be known among experts, however, we were unable to find a proof of its exact statement in the literature, which is why we provide one here.
\begin{lemma}\label{lem:not-lc}
  Let $X$ be a $ \mathbb{Q} $-Gorenstein variety and $x \in X$ a closed point. If there is no log canonical neighborhood of $x$  in $X$, then $\mld_{x}\left( X \right) \leq \dim \left( X \right)$. If $X$ is a $T$-variety, the statement holds with $T$-invariant neighborhoods.
\end{lemma}
\begin{proof}
  We will assume that $\dim \left( X \right) \geq 3$, which is no loss of generality, since Shokurov's conjecture has already been proven for surfaces.\\
  By Theorem 8.11 \cite{Vil2014Equimultiplicity}, there exists a log resolution of $X$, which can be chosen $T$-equivariantly if $X$ is a $T$-variety. Let $f \colon Y \to X$ denote this resolution and let $D_{1}, \dots, D_{s}$ be the $f$-exceptional prime divisors. Let $\Delta = \sum_{i}^{} -\d_{D_{i}}\left( X \right) \cdot D_{i}$ be the discrepancy divisor on $Y$. It follows from Corollary 2.31 (3) in \cite{KoMo1998BirGeom} that $X$ is log canonical if and only if $\d_{D_{i}} \geq -1$ for all $i$.\\
  Discarding from $X$ the centers of all those $D_{i}$ for which $x \notin c_{X}\left( D_{i} \right)$ results in a ($T$-invariant) open subset of $X$ and does not change $\mld_{x}\left( X \right)$. Hence we may assume that $x \in c_{X} \left( D_{i} \right)$ for all $i$.\\
  Suppose now that $X$ is not log canonical, by reordering if necessary, we may assume that $E = D_{1}$ is an $f$-exceptional prime divisor with discrepancy $\d_{E}\left( X \right) = -1-c$ for some $c > 0$. It is well known that if a variety $X$ is not log canonical, then $\mld_{X}\left( X \right) = - \infty$, see for instance Corollary 2.31 (1) in \cite{KoMo1998BirGeom}. We will slightly modify the proof of that Lemma to show that if $X$ is not log canonical, and $x  \in c_{X}\left( E \right)$, where $E$ is the divisor from above with discrepancy $-1-c$, then $\mld_{x}\left( X \right) \leq \dim \left( X \right)$, which is sufficient for our purposes.\\
  Let $Z_{0} \subseteq E$ be a codimension two subset of $Y$ that contains a point $y \in f^{-1}\left( x \right)$ and that is not contained in any of the $f$-exceptional divisors except $E$. By shrinking, we may assume that $E$ and $Z_{0}$ are smooth in $Y$. 
  Let $f_{1} \colon Y_{1} \to Y$ be the blow-up of $Y$ in $Z_{0}$ and $E_{1}$ its exceptional divisor. Then $\d_{E_{1}}\left( Y, \Delta \right) = \d_{E_{1}}\left( X \right) = 2 -1 - c - 1 = - c$ by the blow-up formula for discrepancies.\\
  Let $L = \sum_{E' \ne E}^{} \d_{E'}\left( X \right)$, where the sum is taken over all $f$-exceptional divisors different from $E$. If $L > 0$, then, proceeding as in the proof of Corollary 2.31 (1) in \cite{KoMo1998BirGeom}, we repeatedly blow up in a smooth, codimension two center to obtain after $\ell $ steps an exceptional divisor $E_{\ell}$ with negative discrepancy $\d_{E_{\ell}}\left( X \right) \leq -L$. If $L \leq 0$, we can take $\ell = 0$ and $E_{0} = E$.\\
  Now we blow up the intersection of $E_\ell$ with the strict transform of $E$ and the inverse image of $x$, of which we may assume after shrinking $Y_{\ell}$ that it is smooth, and obtain a birational morphism $Y_{\ell +1} \to Y_{\ell}$ with exceptional divisor $E_{\ell +1}$. The center of $E_{\ell +1}$ over $X$  is $x$, and the log discrepancy of $E_{\ell +1}$ satisfies 
  \begin{align*}
    a_{E_{\ell +1}}\left( X \right) &\leq \dim \left( X \right) + \d_{E_{\ell}}\left( X \right) + \sum_{E' \ne E}^{} \d_{E'}\left( X \right)\\
                                    &\leq \dim \left( X \right)
  \end{align*}
  by the choice of $\ell$ above. In particular, $\mld_{x}\left( X \right) \leq \dim \left( X \right)$ follows.
\end{proof}

\subsection{The lower semi-continuity conjecture}\label{sec:lsc}
A conjecture that is closely related to Shokurov's conjecture about minimal log discrepancies is the following one, proposed by Ambro in \cite{Am1999OnMLD}. It is commonly referred to as the lower semi-continuity or LSC conjecture.
\begin{conjecture}\label{con:lsc}
  Let $X$ be a $ \mathbb{Q} $-Gorenstein variety. Then the function 
  \[
    \mld \colon X  \to \mathbb{R} \cup \left\{ -\infty\right\},\quad x \mapsto  \mld_{x}\left( X \right),
  \]
  defined on the closed points of $X$, is lower semi-continuous, that is, for any $x \in X$, there exists an open neighborhood $U $ of $x$ such that $\mld_{x}\left( X \right) \leq \mld_{y}\left( X \right)$ for all $y \in U$.
\end{conjecture}
In the same work, Ambro showed that the LSC conjecture is equivalent to a stronger version of the inequality proposed by Shokurov's conjecture, in particular, it implies Shokurov's conjecture. Moreover, he showed that the LSC conjecture holds for toric varieties and in dimension at most three.\\
The LSC conjecture was further proven for smooth varieties in \cite{EiMuYa2003LogDis} and for local complete intersections in \cite{EiMu2004InversionOfAdj}.\\
Knowing all this, we assume that the LSC conjecture should hold for klt general arrangement varieties as well, however the tools and results that we developed in this article do not seem to immediately imply this.

\appendix

\section{Polyhedral divisors}\label{sec:p-divs}
In this appendix, we will briefly introduce the language of polyhedral divisors and prove the open claims from Section \ref{sec:mld-bound}. 
For a more detailed introduction to the subject, we recommend the survey article \cite{AIOPSV2012GeomOfTVar} as a starting point.\\
Let $N$ be a lattice, i.e. a free abelian group of finite rank. Let $M$ be the dual lattice $\Hom_{\mathbb{Z}}\left( N,\mathbb{Z} \right)$ of $N$, and denote by $\left\langle \cdot,\cdot \right\rangle$ the natural pairing between $M$ and $N$.\\
The $\mathbb{Q}$-vector space $N \otimes_{\mathbb{Z}} \mathbb{Q}$ spanned by a lattice $N$ will be denoted by $N_{\mathbb{Q}}$. By a \emph{polyhedron} in $N_{\mathbb{Q}}$, we denote a convex rational polyhedron in $N_{\mathbb{Q}}$, i.e. a subset of $N_{\mathbb{Q}}$ that is an intersection of a finite number of closed half-spaces in $N_{\mathbb{Q}}$.\\
Let $\Delta \subseteq N_{\mathbb{Q}}$ be a polyhedron and $H \subseteq N_{\mathbb{Q}}$ a closed half-space such that $\Delta \subseteq H$. The intersection $\Delta' = H \cap \Delta$ is called a \emph{face} of $\Delta$. We use the shorthand $\Delta \left( n \right)$ to denote the set of $n$-dimensional faces of $\Delta$.\\
By a \emph{cone} in $N_{\mathbb{Q}}$, we denote a convex rational polyhedral cone in $N_{\mathbb{Q}}$. The \emph{dual} of a cone $\sigma \subseteq N_{\mathbb{Q}}$ is the set 
\[
  \sigma ^{\vee} = \left\{ u \in M_{\mathbb{Q}}\mid \left\langle u,v \right\rangle \geq 0 \text{ for all }v \in \sigma \right\} \subseteq M_{\mathbb{Q}},
\]
 and this set is in fact a cone in $M_{\mathbb{Q}}$. We call a cone pointed if $\sigma \cap - \sigma = \left\{ 0\right\}$. It is a fact that a cone $\sigma \subseteq N_{\mathbb{Q}}$ is pointed if and only if its dual cone $\sigma^{\vee} \subseteq M_{\mathbb{Q}}$ is full-dimensional.\\
Let $\Delta \subseteq N_{\mathbb{Q}}$ be a polyhedron, the \emph{tail cone} of $\Delta$, denoted $\tail \left( \Delta \right)$, is the cone spanned by all unbounded directions in $\Delta$. Note that the tail cone of a bounded polyhedron is the trivial cone $\left\{0\right\}$. \\
The \emph{$w$-face} of a polyhedron $\Delta$, for a $w \in \tail \left( \Delta \right)^{\vee} \cap M$, is the set
\[
  \face \left( \Delta, w \right) := \left\{v \in \Delta \mid \left\langle w,v \right\rangle \leq \left\langle w, v' \right\rangle \text{ for all } v' \in \Delta\right\},
\]
which is a polyhedron with tail cone $\tail \left( \Delta \right) \cap w^{\perp} $.\\
Let $\sigma \subseteq N_{_\mathbb{Q}}$ be a pointed cone and consider the set of all polyhedra $\Delta $ in $N_{\mathbb{Q}}$ whose tail is equal to $ \sigma$. Define a semigroup structure on this set via Minkowski addition of polyhedra. This set is closed under Minkowski sums and the neutral element is the cone  $\sigma$ itself. We denote the resulting semigroup by $\Pol_{\mathbb{Q}}^{+}\left( N, \sigma\right) $. \\
Let  $Y$ be a normal semiprojective variety (that is, $ Y$ is projective over an affine variety). Let $\cadiv_{\geq 0}\left( Y \right)$ denote the semigroup of effective Cartier divisors on $Y$ under addition. A \emph{polyhedral divisor} $\mathcal{D}$ on $\left( Y, N \right)$ is a finite formal sum
\[
  \mathcal{D} = \sum_{Z}^{} \Delta_{Z} \otimes Z \in \Pol_{\mathbb{Q}}^{+}\left( N, \sigma\right)  \otimes \cadiv_{\geq 0}\left( Y \right).
\]
Hence a polyhedral divisor is a formal sum of Cartier divisors on $Y$ whose coefficients, instead of numbers, are polyhedra in $N_{\mathbb{Q}}$ with a common tail cone equal to $\sigma$. We call the cone $\sigma$ the \emph{tail cone} of $\mathcal{D}$, denoted $\tail \left( \mathcal{D} \right)$. \\
Note that we allow the empty set as a coefficient. The \emph{locus} of $\mathcal{D}$ is the set $\Loc \mathcal{D} = Y \setminus \bigcup_{\Delta_{Z} = \emptyset}^{}Z$. We may abbreviate this to $\Loc$ to make the notation simpler if $\mathcal{D}$ is clear from the context. The \emph{support} of a polyhedral divisor $\mathcal{D}$ as above is the set of all $Z \in \cadiv \left( Y \right)$ with coefficient $\Delta_Z$ different from $\sigma$.\\
For an element $u \in \tail\left( \mathcal{D} \right)^{\vee}$, we define the evaluation of $\mathcal{D}$ at $u$,
\[
  \mathcal{D} \left( u \right) := \sum_{Z}^{} \min \left\langle \Delta_{Z},u \right\rangle \cdot Z_{\mid \Loc } ,
\]
which is a Cartier $\mathbb{Q}$-divisor on $\Loc \mathcal{D} \subseteq Y$.
\begin{definition}
  Let $\mathcal{D}$ be a polyhedral divisor on $\left( \mathbb{P}^{d}, N \right)$ such that the Cartier divisors in the support of $\mathcal{D}$ form a general arrangement of hyperplanes in $\mathbb{P}^{d}$. We will call such $\mathcal{D}$ a \emph{general arrangement} polyhedral divisor.
\end{definition}
Let $\mathcal{D} = \sum_{Z}^{} \Delta_{Z} \otimes Z$ be a general arrangement polyhedral divisor. Define the \emph{degree} of $\mathcal{D}$ as 
\[
  \deg \mathcal{D} = \sum_{Z}^{} \Delta_{Z},
\]
  which is a polyhedron in $N_{\mathbb{Q}}$ with tail cone $\sigma$.
\begin{definition}
Let $\mathcal{D}$ be a polyhedral divisor on $\left( Y,N \right)$. We call $\mathcal{D}$ a \emph{$p$-divisor} or \emph{proper} if the evaluations $\mathcal{D} \left( u \right)$ are semiample for all $u \in \sigma^{\vee} $ and, additionally, $\mathcal{D}\left( u \right)$ is big for all $u \in  \int \left( \sigma ^{\vee}  \right)$, where $\int$ denotes the relative interior of a polyhedron.
\end{definition}
\begin{remark}
  For a general arrangement polyhedral divisor $\mathcal{D}$, there is the following simpler characterization of properness: $\mathcal{D}$ is proper if and only if $\deg \left( \mathcal{D} \right) \subsetneq \tail \left( \mathcal{D} \right)$ holds.
\end{remark}
Let $\mathcal{D}$ be a p-divisor on $\left( Y,N \right)$. We associate to $\mathcal{D}$ a sheaf of graded algebras as follows:
for $u \in \sigma^{\vee} \cap M$, let $\mathcal{A}_{u} = \mathcal{O}_{\Loc} \left( \mathcal{D}\left( u \right) \right)\chi^{u}$ and 
\[
  \mathcal{A} = \bigoplus_{u \in \sigma^{\vee} \cap M}^{} \mathcal{A}_{u}.
\]
We also associate the graded algebra 
\[
  A = \bigoplus_{u \in \sigma^{\vee} \cap M}^{} \Gamma \left( \Loc \left( \mathcal{D} \right), \mathcal{O}_{\Loc}\left( \mathcal{D}\left( u \right) \right) \right)
\]
to $\mathcal{D}$ and define the varieties
\[
  \widetilde{\TV} \left( \mathcal{D} \right) = \Spec_{Y}\left( \mathcal{A} \right) \quad \text{and}  \quad \TV \left( \mathcal{D} \right) = \Spec \left( A \right).
\]
By Theorems 3.1 and 3.4 in \cite{AlHa2006PolyAndTorus}, the varieties thus obtained are $T$-varieties and moreover, all affine $T$-varieties arise as $\TV \left( \mathcal{D} \right)$ for some $p$-divisor $\mathcal{D}$ over a normal semiprojective variety $Y$.\\
We now have all the background needed to describe general arrangement varieties in terms of $p$-divisors:
\begin{proposition}
  Let $\mathcal{D}$ be a general arrangement $p$-divisor on $\left( \mathbb{P}^{d},N \right)$. The corresponding affine $T$-variety $\TV \left( \mathcal{D} \right)$ is a general arrangement variety of complexity $d$ and moreover, all affine general arrangement varieties arise in this way.
\end{proposition}
\begin{proof}
  This follows directly from Remark 6.19 in \cite{HaHiWr2019OnTorusAct}.
\end{proof}
It is possible to characterize $T$-invariant open embeddings between affine $T$-varieties in terms of the corresponding $p$-divisors, but in general, the conditions on the $p$-divisors are quite complicated. However, for general arrangement varieties, there is a rather straightforward combinatorial criterion in terms of the $p$-divisors to characterize open embeddings. Lemma 6.7 in \cite{AlHaSu2008Gluing} tells us the following:\\
Let $\mathcal{D}_{1} = \sum_{Z}^{} \Delta^{\left( 1 \right)}_{Z}\otimes Z$ and $\mathcal{D}_{2}= \sum_{Z}^{} \Delta_{Z}^{\left( 2 \right)}\otimes Z$ be two $p$-divisors on $\left( Y,N \right)$ with the same support and such that $\mathcal{D}_{1}$ has non-empty coefficients. Let $u\in \tail \left( \mathcal{D}_{1} \right)^{\vee} $ be such that every coefficient $\Delta^{\left( 2 \right)}_{Z}$ of $\mathcal{D}_{2}$ is the $u$-face of the corresponding coefficient $\Delta_{Z}^{\left( 1 \right)}$ of $\mathcal{D}_{1}$. Then the induced morphism of $T$-varieties $\TV \left( \mathcal{D}_{2}\right) \to \TV \left( \mathcal{D}_{1} \right)  $ is an open embedding if and only if 
\[
  \forall y \in Y \exists D_{y}\in | \mathcal{D}_{1}\left( \mathbb{N}u \right) | : y \notin \supp D_{y}.
\]
In our setting, since $\mathcal{D}_{1}$ is a $p$-divisor on $\left( \mathbb{P}^{d},N \right)$, $\mathcal{D}_{1}\left( u \right)$ is semiample for all $u\in \tail \left( \mathcal{D}_{1} \right)^{\vee} $, hence this condition is certainly fulfilled. Hence the face relations on the coefficients of $\mathcal{D}_{1}$ and $\mathcal{D}_{2}$ described just above are sufficient to guarantee an open embedding $\TV \left( \mathcal{D}_{2} \right)\to \TV \left( \mathcal{D}_{1} \right)$.\\
Finally, we can state and prove the lemma used to treat some of the special cases in Section \ref{sec:mld-bound}.
\begin{lemma}\label{lem:not_cone}
  Let $X$ be an affine general arrangement variety of complexity $d$ that is not a cone singularity. Then either $X$ is toroidal, or otherwise $X$ can be embedded $T$-equivariantly as an open subset of a general arrangement variety of the same complexity and dimension and which is a cone singularity.
\end{lemma}
\begin{proof}
  We start by pointing out that an affine $T$-variety $X = \TV \left( \mathcal{D} \right)$ is a cone singularity if and only if the only invariant regular functions on $X$ are the constant functions and the tail cone of the associated $p$-divisor $\mathcal{D}$ has a full-dimensional tail cone $\sigma$ in $N_{\mathbb{Q}}$. To see this, recall that by Remark \ref{rem:cone}, $X$ is a cone singularity if and only if $X$ has a $T$-invariant embedding into an affine space such that the torus $T$ acts with positive weights on the coordinates. The positive weights imply that the weight cone for the torus action, which is the dual cone $\sigma^{\vee}$, is pointed, which implies that $\sigma$ is full-dimensional. Moreover, the positive weights directly imply that the only $T$-invariant regular functions on $X$ are the constant ones.\\
 For the reverse implication, observe that if the only $T$-invariant regular functions on $X$ are the constants, then no weight of the action may be zero. Moreover, if $\sigma $ is full-dimensional, then its dual cone is pointed, hence after a change of coordinates, we may assume that all weights of the action are positive.\\
  Now, let $X $ be as in the hypothesis. 
  Let $\mathcal{D}$ be a $p$-divisor on $\left( \mathbb{P}^{d},N \right)$ such that $X = \TV \left( \mathcal{D} \right)$ and let $\sigma = \tail \left( \mathcal{D} \right)$.\\
  If $X$ has non-constant invariant regular functions, then $\Loc \left( \mathcal{D} \right) $ must be affine. It follows that $X = \widetilde{\TV} \left( \mathcal{D} \right)$ and by Proposition 2.6 in \cite{LiSu2013NormalSing}, the variety $\widetilde{\TV} \left( \mathcal{D} \right)$ is toroidal, i.e. every point $x \in \widetilde{\TV} \left( \mathcal{D} \right)$ has a formal neighborhood that is isomorphic to a formal neighborhood of a point in an affine toric variety.\\
  Otherwise, the only invariant regular functions on $X$ are the constants, but the tail cone $\sigma = \tail \mathcal{D} \subseteq N_{\mathbb{Q}}$ is not full-dimensional. We will construct a cone singularity $X'$ to embed $X$ into as an open subset. \\
  Take a basis of the subspace of $N_{\mathbb{Q}}$ that is spanned by $\sigma$ and extend it to a basis of $N_{\mathbb{Q}}$ by vectors $e_{r+1}, \dots, e_{s}$. Let $\sigma'$ be the Minkowski sum $\sigma\hspace{0.10em} + \hspace{0.10em} \Cone \left(e_{r+1}, \dots, e_{s} \right)$.
  Defined like this, $\sigma'$ is a full-dimensional pointed cone in $N_\mathbb{Q}$. For every polyhedral coefficient $\Delta_{Z}$ of $\mathcal{D}$, we define $\Delta'_{Z} = \Delta_{Z} + \Cone \left( e_{r+1}, \dots, e_{s} \right)$. We define the polyhedral divisor $\mathcal{D}'=\sum_{Z}^{} \Delta'_{Z} \otimes  Z$ with tail cone $\sigma'$. Its degree is equal to $\deg \left( \mathcal{D} \right)\hspace{0.10em}+\hspace{0.10em} \Cone \left( e_{r+1}, \dots, e_{s} \right)$, thus the properness of $\mathcal{D}'$ immediately follows from by the properness of $\mathcal{D}$.\\
  By construction, $\mathcal{D}$ and $\mathcal{D}'$ have the same support and the face relation $\Delta_{Z} = \face \left( e_{r+1}, \face \left( e_{r+2}, \dots \face\left( e_{s},\Delta'_{Z} \right) \right) \right)$ holds at every $Z \in \supp \left( \mathcal{D}' \right)$, hence we can identify $X$ with an open $T$-invariant subset of $\TV \left( \mathcal{D}' \right)$ by the observations above the Lemma. The latter variety is is a cone singularity of complexity $d$, since the tail cone of $\mathcal{D}'$ is full-dimensional and the invertible regular functions of $\TV \left( \mathcal{D}' \right)$ are the same as those of $X$, which are in particular only the constant ones. 
\end{proof}
\bibliographystyle{alpha}
\bibliography{../../sources/sourcefile}
\end{document}